\documentclass[12pt,fleqn]{article}

\usepackage{amsmath,amssymb,amsthm}
\usepackage[margin=1in]{geometry}
\usepackage[pdfpagemode=UseNone,pdfstartview=FitH,hypertex]{hyperref}

\newcommand{\A}{{\mathcal A}}
\newcommand{\bgset}[1]{\big\{#1\big\}}
\newcommand{\dint}{\ds{\int}}
\newcommand{\ds}[1]{\displaystyle #1}
\newcommand{\hquad}{\hspace{0.08in}}
\newcommand{\ip}[3][]{\left(#2,#3\right)_{#1}}
\newcommand{\M}{{\cal M}}
\newcommand{\N}{\mathbb N}
\newcommand{\norm}[2][]{\left\|#2\right\|_{#1}}
\renewcommand{\O}{\text{O}}
\renewcommand{\o}{\text{o}}
\newcommand{\pnorm}[2][]{\if #1'' \left|#2\right|_p \else \left|#2\right|_{#1} \fi}
\newcommand{\pow}{{(p-2)/p}}
\newcommand{\R}{\mathbb R}
\newcommand{\restr}[2]{\left.#1\right|_{#2}}
\newcommand{\set}[1]{\left\{#1\right\}}
\newcommand{\wop}{{p/(p-2)}}

\newenvironment{enumroman}{\begin{enumerate}

}{\end{enumerate}}

\newtheorem{corollary}{Corollary}[section]
\newtheorem{lemma}[corollary]{Lemma}
\newtheorem{proposition}[corollary]{Proposition}
\newtheorem{theorem}[corollary]{Theorem}

\numberwithin{equation}{section}

\title{\bf A multiplicity result for the scalar field equation\thanks{{\em MSC2010:} Primary 35J61, 35P30, Secondary 35J20
\newline \indent\; {\em Key Words and Phrases:} scalar field equation, multiple nontrivial solutions, variational and minimax methods, concentration compactness, symmetry breaking}}
\author{\bf Kanishka Perera\\
Department of Mathematical Sciences\\
Florida Institute of Technology\\
Melbourne, FL 32901\\
\em kperera@fit.edu}
\date{}

\begin{document}

\maketitle

\begin{abstract}
We prove the existence of $N - 1$ distinct pairs of nontrivial solutions of the scalar field equation in $\R^N$ under a slow decay condition on the potential near infinity, without any symmetry assumptions. Our result gives more solutions than the existing results in the literature when $N \ge 6$. When the ground state is the only positive solution, we also obtain the stronger result that at least $N - 1$ of the first $N$ minimax levels are critical, i.e., we locate our solutions on particular energy levels with variational characterizations. Finally we prove a symmetry breaking result when the potential is radial. To overcome the difficulties arising from the lack of compactness we use the concentration compactness principle of Lions, expressed as a suitable profile decomposition for critical sequences.
\end{abstract}

\section{Introduction}

Consider the eigenvalue problem for the scalar field equation
\begin{equation} \label{1.1}
- \Delta u + V(x)\, u = \lambda\, |u|^{p-2}\, u, \quad u \in H^1(\R^N),
\end{equation}
where $N \ge 2$, $V \in L^\infty(\R^N)$ satisfies
\begin{equation} \label{1.2}
\lim_{|x| \to \infty} V(x) = V^\infty > 0,
\end{equation}
$p \in (2,2^\ast)$, and $2^\ast = 2N/(N-2)$ if $N \ge 3$ and $2^\ast = \infty$ if $N = 2$. Let
\[
I(u) = \int_{\R^N} |u|^p, \quad J(u) = \int_{\R^N} |\nabla u|^2 + V(x)\, u^2, \quad u \in H^1(\R^N).
\]
Then the eigenfunctions of \eqref{1.1} on the manifold
\[
\M = \bgset{u \in H^1(\R^N) : I(u) = 1}
\]
and the corresponding eigenvalues coincide with the critical points and the corresponding critical values of the constrained functional $\restr{J}{\M}$, respectively. Equation \eqref{1.1} has been studied extensively for more than three decades (see Bahri and Lions \cite{MR1450954} for a detailed account). The main difficulty here is the lack of compactness inherent in this problem. This lack of compactness originates from the invariance of $\R^N$ under the action of the noncompact group of translations, and manifests itself in the noncompactness of the Sobolev imbedding $H^1(\R^N) \hookrightarrow L^p(\R^N)$. This in turn implies that the manifold $\M$ is not weakly closed in $H^1(\R^N)$ and that $\restr{J}{\M}$ does not satisfy the usual Palais-Smale compactness condition at all energy levels.

Least energy solutions, also called ground states, are well-understood. In general, the infimum
\[
\lambda_1 := \inf_{u \in \M}\, J(u)
\]
is not attained. For the autonomous problem at infinity,
\begin{equation} \label{1.3}
- \Delta u + V^\infty\, u = \lambda\, |u|^{p-2}\, u, \quad u \in H^1(\R^N),
\end{equation}
the corresponding functional
\[
J^\infty(u) = \int_{\R^N} |\nabla u|^2 + V^\infty\, u^2
\]
attains its infimum
\[
\lambda_1^\infty := \inf_{u \in \M}\, J^\infty(u) > 0
\]
at a radial function $w_1^\infty > 0$ and this minimizer is unique up to translations (see Berestycki and Lions \cite{MR695535} and Kwong \cite{MR969899}). For the nonautonomous problem, we have $\lambda_1 \le \lambda_1^\infty$ by \eqref{1.2} and the translation invariance of $J^\infty$, and $\lambda_1$ is attained if $\lambda_1 < \lambda_1^\infty$ (see Lions \cite{MR778970,MR778974}).

As for higher energy solutions, also called bound states, radial solutions have been extensively studied when the potential $V$ is radially symmetric (see, e.g., Berestycki and Lions \cite{MR695536}, Grillakis \cite{MR1054554}, Bartsch and Willem \cite{MR1237913}, and Conti et al. \cite{MR1773818}). The subspace $H^1_r(\R^N)$ of $H^1(\R^N)$ consisting of radially symmetric functions is compactly imbedded into $L^p(\R^N)$ for $p \in (2,2^\ast)$ by a compactness result of Strauss \cite{MR0454365}, so in this case the restrictions of $J$ and $J^\infty$ to $\M \cap H^1_r(\R^N)$ have increasing and unbounded sequences of critical values given by a standard minimax scheme. Furthermore, Sobolev imbeddings remain compact for subspaces with any sufficiently robust symmetry (see, e.g., Bartsch and Willem \cite{MR1244943}, Bartsch and Wang \cite{MR1349229}, and Devillanova and Solimini \cite{MR2895950}).

As for multiplicity in the nonsymmetric case, Zhu \cite{MR969507}, Hirano \cite{MR1866736}, and Clapp and Weth \cite{MR2103845} have given sufficient conditions for the existence of $2$, $3$, and $N/2 + 1$ pairs of solutions, respectively (see also Li \cite{MR1205151}). Let us also mention that Cerami et al. \cite{MR2138080} have obtained infinitely many solutions under considerably strong additional asymptotic assumptions. There is also an extensive literature on multiple solutions of scalar field equations in topologically nontrivial unbounded domains (see the survey paper of Cerami \cite{MR2278729}). In the present paper we obtain $N - 1$ pairs of solutions in the whole space, without any symmetry assumptions. We assume that
\[
W := V^\infty - V \in L^\wop(\R^N),
\]
and write $\pnorm[q]{\cdot}$ for the norm in $L^q(\R^N)$. Our multiplicity result is the following.

\begin{theorem} \label{Theorem 1.1}
Assume that $N \ge 3$, $V \in L^\infty(\R^N)$ satisfies \eqref{1.2}, $p \in (2,2^\ast)$, and $W \in L^\wop(\R^N)$ satisfies
\begin{equation} \label{1.4}
\pnorm[\wop]{W} < \left(2^\pow - 1\right) \lambda_1^\infty
\end{equation}
and
\begin{equation} \label{1.5}
W(x) \ge c_0\, e^{- a\, |x|} \quad \forall x \in \R^N
\end{equation}
for some constants $0 < a < 2\, \sqrt{V^\infty}$ and $c_0 > 0$. Then the equation \eqref{1.1} has $N - 1$ pairs of eigenfunctions on $\M$.
\end{theorem}

Our result gives more solutions than \cite{MR2103845} when $N \ge 6$. Moreover, our proof is simpler than that in \cite{MR2103845} and does not involve any dynamical systems theory arguments. Note also that we do not assume that $V$ is a positive function as in \cite{MR2103845}.

We obtain a stronger result when $\lambda_1$ is the only eigenvalue of \eqref{1.1} with a positive eigenfunction on $\M$. Let $\A$ denote the class of all nonempty closed symmetric subsets of $\M$, let
\[
\gamma(A) = \inf\, \bgset{l \ge 1 : \exists \text{ an odd continuous map } A \to \R^l \setminus \{0\}}
\]
be the genus of $A \in \A$, and set
\[
\lambda_j := \inf_{\substack{A \in \A\\[1pt]
\gamma(A) \ge j}}\, \sup_{u \in A}\, J(u), \quad \lambda_j^\infty := \inf_{\substack{A \in \A\\[1pt]
\gamma(A) \ge j}}\, \sup_{u \in A}\, J^\infty(u), \quad j \ge 2.
\]
We have $\lambda_j \le \lambda_j^\infty$ by \eqref{1.2} and the translation invariance of $J^\infty$, and it is known that
\begin{equation} \label{1.6}
\lambda_j^\infty = 2^\pow\, \lambda_1^\infty, \quad j = 2,\dots,N
\end{equation}
(see Perera and Tintarev \cite{PeTi}), so
\begin{equation} \label{1.7}
\lambda_1 \le \cdots \le \lambda_N \le 2^\pow\, \lambda_1^\infty.
\end{equation}
Under the hypotheses of Theorem \ref{Theorem 1.1}, $\lambda_1 < \lambda_1^\infty$ and hence $\lambda_1$ is an eigenvalue of \eqref{1.1}, and it was recently shown in Perera and Tintarev \cite{PeTi} that $\lambda_2$ is also an eigenvalue.

\begin{theorem} \label{Theorem 1.2}
Assume that $N \ge 3$, $V \in L^\infty(\R^N)$ satisfies \eqref{1.2}, $p \in (2,2^\ast)$, and $W \in L^\wop(\R^N)$ satisfies \eqref{1.4} and \eqref{1.5}. If \eqref{1.1} has no positive eigenfunctions on $\M$ corresponding to eigenvalues in $(\lambda_1,\lambda_1^\infty)$, and has only a finite number of eigenfunctions on $\M$ corresponding to eigenvalues in $(\lambda_1,2^\pow\, \lambda_1^\infty)$, then at least $N - 1$ of the minimax levels $\lambda_1,\dots,\lambda_N$ are eigenvalues of \eqref{1.1}.
\end{theorem}

Finally we prove a symmetry breaking result when $V$ is radial. Let $\A_r$ denote the class of all nonempty closed symmetric subsets of $\M_r = \M \cap H^1_r(\R^N)$ and set
\[
\lambda_{j,\, r} := \inf_{\substack{A \in \A_r\\[1pt]
\gamma(A) \ge j}}\, \sup_{u \in A}\, J(u), \quad \lambda_{j,\, r}^\infty := \inf_{\substack{A \in \A_r\\[1pt]
\gamma(A) \ge j}}\, \sup_{u \in A}\, J^\infty(u), \quad j \ge 1.
\]
Since the imbedding $H^1_r(\R^N) \hookrightarrow L^p(\R^N)$ is compact, these radial minimax levels are critical for $\restr{J}{\M_r}$ and $\restr{J^\infty}{\M_r}$, respectively. We have $\lambda_{1,\, r} = \lambda_1$ and $\lambda_{1,\, r}^\infty = \lambda_1^\infty$. In general, $\lambda_j \le \lambda_{j,\, r}$ and $\lambda_j^\infty \le \lambda_{j,\, r}^\infty$, and it is known that $\lambda_2^\infty$ is not critical for $\restr{J^\infty}{\M}$ (see, e.g., Weth \cite{MR2263672}), so $\lambda_2^\infty < \lambda_{2,\, r}^\infty$.

\begin{theorem} \label{Theorem 1.3}
Assume that $N \ge 3$, $V \in L^\infty(\R^N)$ is radial and satisfies \eqref{1.2}, $p \in (2,2^\ast)$, and $W \in L^\wop(\R^N)$ satisfies \eqref{1.4}, \eqref{1.5}, and
\begin{equation} \label{1.8}
\pnorm[\wop]{W} \le \lambda_{2,\, r}^\infty - \lambda_2^\infty.
\end{equation}
Then the equation \eqref{1.1} has $N - 2$ pairs of eigenfunctions on $\M$ corresponding to eigenvalues in $(\lambda_{1,\, r},\lambda_{2,\, r})$. If, in addition, \eqref{1.1} has no positive eigenfunctions on $\M$ corresponding to eigenvalues in $(\lambda_{1,\, r},\lambda_{1,\, r}^\infty)$, and has only a finite number of eigenfunctions on $\M$ corresponding to eigenvalues in $(\lambda_{1,\, r},\lambda_2^\infty)$, then at least $N - 2$ of the minimax levels $\lambda_2,\dots,\lambda_N$ are eigenvalues of \eqref{1.1} in $(\lambda_{1,\, r},\lambda_{2,\, r})$.
\end{theorem}

Our proofs will use the concentration compactness principle of Lions \cite{MR778970,MR778974,MR879032}, expressed as a suitable profile decomposition for critical sequences of $\restr{J}{\M}$, to overcome the difficulties arising from the lack of compactness.

\section{Preliminaries}

We will use the norm
\[
\norm{u} = \sqrt{J^\infty(u)}
\]
on $H^1(\R^N)$, which is equivalent to the standard norm. In the absence of a compact Sobolev imbedding, the main technical tool we use here for handling the convergence matters is the following profile decomposition of Solimini \cite{MR1340267} for bounded sequences in $H^1(\R^N)$.

\begin{lemma} \label{Lemma 2.1}
Let $u_k \in H^1(\R^N)$ be a bounded sequence, and assume that there is a constant $\delta > 0$ such that if $u_k(\cdot + y_k) \rightharpoonup w \ne 0$ on a renumbered subsequence for some $y_k \in \R^N$ with $|y_k| \to \infty$, then $\norm{w} \ge \delta$. Then there are $m \in \N$, $w^{(n)} \in H^1(\R^N)$, $y^{(n)}_k \in \R^N,\, y^{(1)}_k = 0$ with $k \in \N$, $n \in \set{1,\dots,m}$, $w^{(n)} \ne 0$ for $n \ge 2$, such that, on a renumbered subsequence,
\begin{gather}
u_k(\cdot + y^{(n)}_k) \rightharpoonup w^{(n)}, \label{2.1}\\[12.5pt]
\big|y^{(n)}_k - y^{(l)}_k\big| \to \infty \text{ for } n \ne l, \notag\\[10pt]
\sum_{n=1}^m\, \norm{w^{(n)}}^2 \le \liminf \norm{u_k}^2, \notag\\[2.5pt]
u_k - \sum_{n=1}^m\, w^{(n)}(\cdot - y^{(n)}_k) \to 0 \text{ in } L^p(\R^N) \quad \forall p \in (2,2^\ast). \label{2.2}
\end{gather}
\end{lemma}

Recall that $u_k \in \M$ is a critical sequence for $\restr{J}{\M}$ at the level $c \in \R$ if
\begin{equation} \label{2.3}
J'(u_k) - \mu_k\, I'(u_k) \to 0, \qquad J(u_k) \to c
\end{equation}
for some sequence $\mu_k \in \R$. By the H\"older inequality,
\begin{equation} \label{2.4}
|J^\infty(u) - J(u)| \le \int_{\R^N} |W(x)|\, u^2 \le \pnorm[\wop]{W} \pnorm{u}^2 = \pnorm[\wop]{W} \quad \forall u \in \M,
\end{equation}
so $u_k$ is bounded. Equation \eqref{2.3} implies
\begin{equation} \label{2.5}
- \Delta u_k + V(x)\, u_k = c_k\, |u_k|^{p-2}\, u_k + \o(1),
\end{equation}
where $c_k = (p/2)\, \mu_k \to c$ since $\ip{J'(u_k)}{u_k} = 2\, J(u_k)$ and $\ip{I'(u_k)}{u_k} = p\, I(u_k) = p$. So if $u_k(\cdot + y_k) \rightharpoonup w$ on a renumbered subsequence for some $y_k \in \R^N$ with $|y_k| \to \infty$, then $w$ solves \eqref{1.3} with $\lambda = c$ by \eqref{1.2}, in particular, $\norm{w}^2 = c \pnorm{w}^p$. If $w \ne 0$, it follows that $c > 0$ and $\norm{w} \ge \big[(\lambda_1^\infty)^p/c\big]^{1/2\, (p-1)}$ since $\norm{w}^2/\pnorm{w}^2 \ge \lambda_1^\infty$. Thus, we have the following profile decomposition of Benci and Cerami \cite{MR898712} for critical sequences of $\restr{J}{\M}$.

\begin{lemma} \label{Lemma 2.2}
Let $u_k \in \M$ be a critical sequence for $\restr{J}{\M}$ at the level $c \in \R$. Then it admits a renumbered subsequence that satisfies the conclusions of Lemma \ref{Lemma 2.1} for some $m \in \N$, and, in addition,
\begin{gather}
- \Delta w^{(1)} + V(x)\, w^{(1)} = c\, |w^{(1)}|^{p-2}\, w^{(1)}, \notag\\[12.5pt]
\qquad - \Delta w^{(n)} + V^\infty\, w^{(n)} = c\, |w^{(n)}|^{p-2}\, w^{(n)}, \quad n = 2,\dots,m, \label{2.6}\\[12.5pt]
J(w^{(1)}) = c\, I(w^{(1)}), \qquad J^\infty(w^{(n)}) = c\, I(w^{(n)}), \quad n = 2,\dots,m, \label{2.7}\\[10pt]
\sum_{n=1}^m\, I(w^{(n)}) = 1, \qquad J(w^{(1)}) + \sum_{n=2}^m\, J^\infty(w^{(n)}) = c, \label{2.8}\\[2.5pt]
u_k - \sum_{n=1}^m\, w^{(n)}(\cdot - y^{(n)}_k) \to 0 \text{ in } H^1(\R^N). \label{2.9}
\end{gather}
\end{lemma}

\begin{proof}
The proof is based on standard arguments and we only sketch it. Equations in \eqref{2.6} follow from \eqref{2.5}, \eqref{2.1}, and \eqref{1.2}, and \eqref{2.7} is immediate from \eqref{2.6}. First equation in \eqref{2.8} is a particular case of Tintarev and Fieseler \cite[Lemma 3.4]{MR2294665}, and the second follows from \eqref{2.7} and the first. The limit \eqref{2.9} follows from \eqref{2.2}, \eqref{2.5}, and the continuity of the Sobolev imbedding.
\end{proof}

By \eqref{2.4},
\[
0 \le \lambda_j^\infty - \lambda_j \le \pnorm[\wop]{W} \quad \forall j \ge 1,
\]
and combining this with \eqref{1.4}, \eqref{1.6}, and \eqref{1.7} gives
\begin{equation} \label{2.10}
0 < \lambda_1 \le \lambda_1^\infty < \lambda_2 \le \cdots \le \lambda_N \le 2^\pow\, \lambda_1^\infty.
\end{equation}
Set
\[
\lambda^\# = \Big[\lambda_1^\wop + (\lambda_1^\infty)^\wop\Big]^\pow \in (\lambda_1^\infty,2^\pow\, \lambda_1^\infty],
\]
and let
\[
\pi(u) = \frac{u}{\pnorm{u}}
\]
be the radial projection of $u \in H^1(\R^N) \setminus \set{0}$ on $\M$.

\begin{lemma} \label{Lemma 2.3}
Assume that $\lambda_1^\infty < \lambda_j < 2^\pow\, \lambda_1^\infty$, let $u_k \in \M$ be a critical sequence of $\restr{J}{\M}$ at the level $\lambda_j$, and consider the profile decomposition of $u_k$ given in Lemma \ref{Lemma 2.2}.
\begin{enumroman}
\item \label{Lemma 2.3.i} If $\lambda_j < \lambda^\#$, then $m = 1$, $u_k \to w^{(1)}$, $w^{(1)}$ is a critical point of $\restr{J}{\M}$, and $J(w^{(1)}) = \lambda_j$.
\item \label{Lemma 2.3.ii} If $\lambda_j \ge \lambda^\#$ and $u_k \not\to w^{(1)}$, then $m = 2$, $w^{(1)} \ne 0$, $\pi(w^{(1)})$ is a critical point of $\restr{J}{\M}$,
    \begin{equation} \label{2.11}
    J(\pi(w^{(1)})) = \Big[\lambda_j^\wop - (\lambda_1^\infty)^\wop\Big]^\pow \in [\lambda_1,\lambda_1^\infty),
    \end{equation}
    and $\pi(w^{(1)})$ has fixed sign.
\end{enumroman}
\end{lemma}

\begin{proof}
Set $t_n = I(w^{(n)})$. Then $t_n \ge 0$ and
\begin{equation} \label{2.12}
\sum_{n=1}^m\, t_n = 1
\end{equation}
by \eqref{2.8}, so each $t_n \in [0,1]$. For $n \ge 2$, $t_n \ne 0$ and $\lambda_1^\infty\, t_n^{2/p} \le J^\infty(w^{(n)}) = \lambda_j\, t_n$ by \eqref{2.7}, so
\begin{equation} \label{2.13}
t_n \ge \left(\frac{\lambda_1^\infty}{\lambda_j}\right)^\wop, \quad n = 2,\dots,m.
\end{equation}
Combining \eqref{2.12} and \eqref{2.13} gives $(m - 1)^\pow\, \lambda_1^\infty \le \lambda_j < 2^\pow\, \lambda_1^\infty$, so $m \le 2$. If $t_1 = 0$, then $m = 2$ and $t_2 = 1$ by \eqref{2.12}, so $w^{(2)}$ is a solution of \eqref{1.3} on $\M$ with $\lambda = \lambda_j$ by \eqref{2.6}, which is a contradiction since $\lambda_1^\infty < \lambda_j < 2^\pow\, \lambda_1^\infty$ (see, e.g., Cerami \cite{MR2278729}). Hence $t_1 \ne 0$, and $\lambda_1\, t_1^{2/p} \le J(w^{(1)}) = \lambda_j\, t_1$ by \eqref{2.7}, so
\begin{equation} \label{2.14}
t_1 \ge \left(\frac{\lambda_1}{\lambda_j}\right)^\wop.
\end{equation}
Combining \eqref{2.12}--\eqref{2.14} now gives
\begin{equation} \label{2.15}
\Big[\lambda_1^\wop + (m - 1)\, (\lambda_1^\infty)^\wop\Big]^\pow \le \lambda_j.
\end{equation}

\ref{Lemma 2.3.i} If $\lambda_j < \lambda^\#$, then $m = 1$ by \eqref{2.15} and hence $t_1 = 1$ by \eqref{2.12}, so $u_k \to w^{(1)}$ by \eqref{2.9} and $w^{(1)}$ is a solution of \eqref{1.1} on $\M$ with $\lambda = \lambda_j$ by \eqref{2.6}.

\ref{Lemma 2.3.ii} If $u_k \not\to w^{(1)}$, then $m \ge 2$ by \eqref{2.9} and hence $m = 2$, and $w^{(1)} \ne 0$ since $t_1 \ne 0$. By \eqref{2.6}, $\pi(w^{(1)})$ and $\pi(w^{(2)})$ are solutions of \eqref{1.1} and \eqref{1.3} on $\M$ with $\lambda = t_1^\pow\, \lambda_j$ and $\lambda = t_2^\pow\, \lambda_j$, respectively. Since $J^\infty(\pi(w^{(2)})) = t_2^\pow\, \lambda_j < 2^\pow\, \lambda_1^\infty$, then $t_2^\pow\, \lambda_j = \lambda_1^\infty$, and combining this with $t_1 + t_2 = 1$, $J(\pi(w^{(1)})) = t_1^\pow\, \lambda_j$, and $\lambda^\# \le \lambda^j < 2^\pow\, \lambda_1^\infty$ gives \eqref{2.11}. Since $J(\pi(w^{(1)})) < \lambda_1^\infty < \lambda_2$ by \eqref{2.10}, $\pi(w^{(1)})$ has fixed sign (see, e.g., Cerami \cite{MR2278729}).
\end{proof}

Proof of the following lemma is similar to that of Clapp and Weth \cite[Lemma 8]{MR2103845} and is therefore omitted (see also Devillanova and Solimini \cite[Lemma 2.4]{MR1966256}).

\begin{lemma} \label{Lemma 2.4}
If $\lambda_1^\infty < \lambda_j = \lambda_{j+1} < 2^\pow\, \lambda_1^\infty$, then $\restr{J}{\M}$ has infinitely many critical points with value $\le \lambda_j$.
\end{lemma}

Theorems \ref{Theorem 1.1} and \ref{Theorem 1.2} will follow from the following proposition.

\begin{proposition} \label{Proposition 2.5}
If $N \ge 3$, $V \in L^\infty(\R^N)$ satisfies \eqref{1.2}, $p \in (2,2^\ast)$, and
\[
0 < \lambda_1 < \lambda_1^\infty < \lambda_2 \le \cdots \le \lambda_N < 2^\pow\, \lambda_1^\infty,
\]
then \eqref{1.1} has $N - 1$ pairs of eigenfunctions on $\M$. If, in addition, \eqref{1.1} has no positive eigenfunctions on $\M$ corresponding to eigenvalues in $(\lambda_1,\lambda_1^\infty)$, and only a finite number of eigenfunctions on $\M$ corresponding to eigenvalues in $(\lambda_1,2^\pow\, \lambda_1^\infty)$, then at least $N - 1$ of the minimax levels $\lambda_1,\dots,\lambda_N$ are eigenvalues of \eqref{1.1}.
\end{proposition}

\begin{proof}
We may assume that $\lambda_2 < \cdots < \lambda_N$ in view of Lemma \ref{Lemma 2.4}. For each $j \in \set{2,\dots,N}$, either $\lambda_j > \lambda_1^\infty$ is an eigenvalue, or
\begin{equation} \label{2.16}
\widetilde{\lambda}_j := \Big[\lambda_j^\wop - (\lambda_1^\infty)^\wop\Big]^\pow < \lambda_1^\infty
\end{equation}
is an eigenvalue with a positive eigenfunction on $\M$ by Lemma \ref{Lemma 2.3}. It follows that at least $N - 1$ of the levels $\widetilde{\lambda}_2 < \cdots < \widetilde{\lambda}_N < \lambda_2 < \cdots < \lambda_N$ are eigenvalues. If $\lambda_1$ is the only eigenvalue $< \lambda_1^\infty$ with a positive eigenfunction on $\M$, then any $\widetilde{\lambda}_j \ne \lambda_1$ is not an eigenvalue by \eqref{2.16}.
\end{proof}

\section{Proofs of Theorems \ref{Theorem 1.1}--\ref{Theorem 1.3}}

In view of Proposition \ref{Proposition 2.5} and \eqref{2.10}, to complete the proofs of Theorems \ref{Theorem 1.1} and \ref{Theorem 1.2}, it only remains to show that $\lambda_1 < \lambda_1^\infty$ and $\lambda_N < 2^\pow\, \lambda_1^\infty$ when \eqref{1.5} holds. We have
\[
\lambda_1 \le J(w_1^\infty) = J^\infty(w_1^\infty) - \int_{\R^N} W(x)\, w_1^\infty(x)^2\, dx \le \lambda_1^\infty - c_0 \int_{\R^N} e^{- a\, |x|}\, w_1^\infty(x)^2\, dx < \lambda_1^\infty.
\]
We will show that there exists an $R > 0$ such that, for the odd continuous map $h$ from the unit sphere $S^{N-1} \subset \R^N$ to $\M$ defined by
\[
h(y) = \frac{w_1^\infty(\cdot + Ry) - w_1^\infty(\cdot - Ry)}{\pnorm{w_1^\infty(\cdot + Ry) - w_1^\infty(\cdot - Ry)}}, \quad y \in S^{N-1},
\]
we have
\begin{equation} \label{3.1}
\sup_{u \in h(S^{N-1})}\, J(u) < 2^\pow\, \lambda_1^\infty.
\end{equation}
Since $\gamma(h(S^{N-1})) \ge \gamma(S^{N-1}) = N$, then
\begin{equation} \label{3.2}
\lambda_N < 2^\pow\, \lambda_1^\infty.
\end{equation}

\newpage

First we prove an elementary inequality.

\begin{lemma} \label{Lemma 3.1}
For all $a, b \in \R$ and $p \ge 2$,
\begin{equation} \label{3.3}
|a + b|^p \ge |a|^p + |b|^p - p\, |a|^{p-1}\, |b| - p\, |a|\, |b|^{p-1}.
\end{equation}
\end{lemma}

\begin{proof}
The inequality is clearly true if $a$ or $b$ is zero, or if $a$ and $b$ have the same sign, so suppose that $0 < |b| \le |a|$ and that $a$ and $b$ have opposite signs. Then \eqref{3.3} is equivalent to
\begin{equation} \label{3.4}
(1 - x)^p \ge 1 + x^p - p\, x - p\, x^{p-1} \quad \forall x \in [0,1],
\end{equation}
where $x = |b/a|$. Let $f(x) = (1 - x)^p - 1 - x^p + p\, x,\, x \in [0,1]$. Then $f(0) = 0$, and $f'(x) = p \left[1 - (1 - x)^{p-1} - x^{p-1}\right] \ge p \left[1 - (1 - x) - x\right] = 0$ since $p - 1 \ge 1$, so $f(x) \ge 0$, from which \eqref{3.4} follows.
\end{proof}

Recall that
\begin{equation} \label{3.5}
w_1^\infty(x) \sim C_0\, \frac{e^{- \sqrt{V^\infty}\, |x|}}{|x|^{(N-1)/2}} \text{ as } |x| \to \infty
\end{equation}
for some constant $C_0 > 0$ (see Gidas et al. \cite{MR544879}).

\begin{lemma} \label{Lemma 3.2}
Let $a < b < 2\, \sqrt{V^\infty}$. Then as $R \to \infty$, uniformly in $y \in S^{N-1}$,
\begin{enumroman}
\item \label{Lemma 3.2.i} $\dint_{\R^N} w_1^\infty(x + Ry)^{q-1}\, w_1^\infty(x - Ry)\, dx = \O(e^{- bR}) \quad \forall q \ge 2$,
\item \label{Lemma 3.2.ii} $J(w_1^\infty(\cdot + Ry) - w_1^\infty(\cdot - Ry)) \le 2 \lambda_1^\infty - \dint_{\R^N} W(x)\, w_1^\infty(x + Ry)^2\, dx + \O(e^{- bR})$,
\item \label{Lemma 3.2.iii} $\pnorm{w_1^\infty(\cdot + Ry) - w_1^\infty(\cdot - Ry)} \ge 2^{1/p} + \O(e^{- bR})$.
\end{enumroman}
\end{lemma}

\begin{proof}
\ref{Lemma 3.2.i} Making the change of variable $x \mapsto x + Ry$ gives
\[
\int_{\R^N} w_1^\infty(x + Ry)^{q-1}\, w_1^\infty(x - Ry)\, dx = \int_{\R^N} w_1^\infty(x)^{q-1}\, w_1^\infty(x - 2Ry)\, dx.
\]
By \eqref{3.5}, $w_1^\infty(x) \le C\, e^{- \sqrt{V^\infty}\, |x|}$ for some $C > 0$, so the integral on the right is bounded by a constant multiple of
\[
\int_{\R^N} e^{- \sqrt{V^\infty}\, [(q - 1)\, |x| + |x - 2Ry|]}\, dx \le \int_{\R^N} e^{- \sqrt{V^\infty}\, |x| - b\, (2R - |x|)/2}\, dx = e^{- bR} \int_{\R^N} e^{- (\sqrt{V^\infty} - b/2)\, |x|}\, dx,
\]
and the last integral is finite since $b < 2\, \sqrt{V^\infty}$.

\ref{Lemma 3.2.ii} We have
\begin{align*}
& \hquad J(w_1^\infty(\cdot + Ry) - w_1^\infty(\cdot - Ry))\\[7.5pt]
= & \hquad J^\infty(w_1(\cdot + Ry)) + J^\infty(w_1^\infty(\cdot - Ry)) - \int_{\R^N} W(x)\, \big(w_1^\infty(x + Ry) - w_1^\infty(x - Ry)\big)^2\, dx\\[5pt]
& \hquad - 2 \int_{\R^N} \big(\nabla w_1^\infty(x + Ry) \cdot \nabla w_1^\infty(x - Ry) + V^\infty\, w_1^\infty(x + Ry)\, w_1^\infty(x - Ry)\big)\, dx\\[5pt]
\le & \hquad 2 \lambda_1^\infty - \int_{\R^N} W(x)\, w_1^\infty(x + Ry)^2\, dx + 2 \pnorm[\infty]{W} \int_{\R^N} w_1^\infty(x + Ry)\, w_1^\infty(x - Ry)\, dx\\[5pt]
& \hquad - 2 \lambda_1^\infty \int_{\R^N} w_1^\infty(x + Ry)^{p-1}\, w_1^\infty(x - Ry)\, dx
\end{align*}
since $w_1^\infty(\cdot + Ry)$ solves \eqref{1.3} with $\lambda = \lambda_1^\infty$, and the last two terms are of the order $\O(e^{- bR})$ by part \ref{Lemma 3.2.i}.

\ref{Lemma 3.2.iii} By Lemma \ref{Lemma 3.1},
\begin{align*}
& \hquad \pnorm{w_1^\infty(\cdot + Ry) - w_1^\infty(\cdot - Ry)}\\[7.5pt]
\ge & \hquad \bigg(\pnorm{w_1^\infty(\cdot + Ry)}^p + \pnorm{w_1^\infty(\cdot - Ry)}^p - p \int_{\R^N} w_1^\infty(x + Ry)^{p-1}\, w_1^\infty(x - Ry)\, dx\\[5pt]
& \hquad - p \int_{\R^N} w_1^\infty(x - Ry)^{p-1}\, w_1^\infty(x + Ry)\, dx\bigg)^{1/p}\\[7.5pt]
= & \hquad \big(2 + \O(e^{- bR})\big)^{1/p}
\end{align*}
by part \ref{Lemma 3.2.i}, and the conclusion follows.
\end{proof}

We are now ready to prove \eqref{3.1}. By \eqref{1.5},
\[
\int_{\R^N} W(x)\, w_1^\infty(x + Ry)^2\, dx = \int_{\R^N} W(x - Ry)\, w_1^\infty(x)^2\, dx \ge c\, e^{- aR} \quad \forall R > 0, y \in S^{N-1}
\]
for some $c > 0$. This together with Lemma \ref{Lemma 3.2} gives
\begin{align*}
\sup_{y \in S^{N-1}}\, J(h(y)) = & \hquad \sup_{y \in S^{N-1}}\, \frac{J(w_1^\infty(\cdot + Ry) - w_1^\infty(\cdot - Ry))}{\pnorm{w_1^\infty(\cdot + Ry) - w_1^\infty(\cdot - Ry)}^2}\\[10pt]
\le & \hquad \frac{2 \lambda_1^\infty - c\, e^{- aR}}{2^{2/p}} + \O(e^{- bR})\\[10pt]
< & \hquad 2^\pow\, \lambda_1^\infty
\end{align*}
if $R$ is sufficiently large, since $a < b$.

Since $\lambda_{1,\, r} < \lambda_2$ by \eqref{2.10}, to complete the proof of Theorem \ref{Theorem 1.3}, it only remains to show that $\lambda_N < \lambda_{2,\, r}$ when \eqref{1.8} holds. By \eqref{2.4}, $\lambda_{2,\, r}^\infty - \lambda_{2,\, r} \le \pnorm[\wop]{W}$, and combining this with \eqref{3.2}, \eqref{1.6}, and \eqref{1.8} gives
\[
\lambda_N < \lambda_2^\infty \le \lambda_{2,\, r}^\infty - \pnorm[\wop]{W} \le \lambda_{2,\, r}.
\]

\def\cdprime{$''$}

\end{document}